\documentclass[reqno]{amsart}%
\usepackage{amssymb}
\usepackage{amsmath}
\usepackage{amsfonts}
\usepackage{graphicx}
\usepackage{epstopdf}
\usepackage{epsfig}
\usepackage{hyperref}%
\usepackage{multicol}
\usepackage{subcaption}
\usepackage{caption}
\providecommand{\U}[1]{\protect\rule{.1in}{.1in}}
\newtheorem{theorem}{Theorem}
\theoremstyle{plain}

\newtheorem{corollary}{Corollary}

\newtheorem{example}{Example}

\newtheorem{lemma}{Lemma}

\numberwithin{equation}{section}
\begin{document}
   \title
{On the Structure of the Generalized Group of Units}

\author{Therrar Kadri and Mohammad El-Hindi}

\address{Therrar Kadri \newline	Department of Pedagog,
	Lebanese University.}\email{therrar.kadri@ul.edu.lb}
	\address{Mohammad Elhindi\newline  Department of Mathematics and Computer
		Science, Faculty of Science, Beirut Arab University, Beirut, Lebanon  } \email{mohammadyhindi98@gmail.com}

\begin{abstract}
Let $R$ be a finite commutative ring with identity and $U(R)$ be its group of
\ units. In 2005, El-Kassar and Chehade presented a ring structure for $U(R)$
and as a consequence they generalized this group of units to the generalized
group of units $U^{k}\left(  R\right)  $ defined iteratively as the group of
the units of $U^{k-1}(R)$, with $U^{1}\left(  R\right) =U(R) $. In this paper,
we examine the structure of this group, when $R=%
\mathbb{Z}
_{n}.$ We find a decomposition of $U^{k}\left(
\mathbb{Z}
_{n}\right)  $ as a direct product of cyclic groups for the general case of
any $k$, and we study when these groups are boolean and trivial. We also show
that this decomposition structure is directly related to the Pratt Tree primes.

\end{abstract}
\maketitle

\section{Introduction}

Let $R$ be a finite commutative ring with identity and let $U(R)$ denote its
group of units. The fundamental theorem of finite abelian groups states that
any finite abelian group is isomorphic to a product of cyclic groups. That
is,
\begin{equation}
U(R)\thickapprox\mathbb{Z}_{n_{1}}\times\mathbb{Z}_{n_{2}}\times
\cdots\times\mathbb{Z}_{n_{i}}. \label{1stdec}%
\end{equation}
The problem of determining the structure of the group of units of any
commutative ring $R$ is an open problem and has received lots of attention.
However, the problem is solved for certain classes for example the ring of
integers modulo $n$, $\mathbb{Z}_{n}$, see \cite{6}, and the factor ring of
Gaussian integer modulo $\beta$, $\mathbb{Z[}\mathbf{i}\mathbb{]}/<\beta>$,
see Cross \cite{1}. Also Smith and Gallian in \cite{9}, solved the problem of
decomposing the group of units of the finite ring $F[x]/<h(x)>$, where $F$ is
a finite field and $h(x)$ is polynomial in $F[x]$.

In 2006, a generalization for the group of units of any finite commutative
ring $R$ with identity, was introduced by El-Kassar and Chehade \cite{4}. They
proved that the group of units of a commutative ring $R$; $U(R)$; supports a
ring structure and this has made it possible to define the second group of
units of $R$ as, $U^{2}(R)=U(U(R))$. Extending this definition to the k-th
level, the k-th group of units is defined as, $U^{k}(R)=U(U^{k-1}(R))$. On the
other hand the decomposition in (\ref{1stdec}) can be generalized so that
$U^{k}(R)\thickapprox U^{k-1}\left(  \mathbb{Z}_{n_{1}}\right)  \times
U^{k-1}\left(  \mathbb{Z}_{n_{2}}\right)  \times\cdots\times U^{k-1}\left(
\mathbb{Z}_{n_{i}}\right)  .$ For example, if we consider $R=\mathbb{Z}%
[\mathbf{i}]/\left\langle p^{n}\right\rangle $, the factor ring of Gaussian
integer modulo $p^{n}$, where $p$ is an odd prime in $\mathbb{Z}$ of the form
$p\equiv3(\operatorname{mod}4)$. Cross \cite{1} determined the structure of
the group of units of $\mathbb{Z}[\mathbf{i}]/\left\langle p^{n}\right\rangle
$ as $U(\mathbb{Z}[\mathbf{i}]/\left\langle p^{n}\right\rangle )\thickapprox
\mathbb{Z}_{p^{n-1}}\times\mathbb{Z}_{p^{n-1}}\times\mathbb{Z}_{p^{2}-1}$.
Thus structure of $U^{k}(\mathbb{Z}[\mathbf{i}]/\left\langle p^{n}%
\right\rangle )$ can be examined through the isomorphism $U^{k}(\mathbb{Z}%
[\mathbf{i}]/\left\langle p^{n}\right\rangle )\thickapprox U^{k-1}\left(
\mathbb{Z}_{p^{n-1}}\right)  \times U^{k-1}\left(  \mathbb{Z}_{p^{n-1}%
}\right)  \times U^{k-1}\left(  \mathbb{Z}_{p^{2}-1}\right)  $. Arising from
all finite commutative rings $R$ with identity, the structure of $U^{k}(R)$ is
obtained through the structure the generalized group of units of
$\mathbb{Z}_{n}$.

Moreover, let $R$ be any finite ring with $\left\vert R\right\vert \,>1$.
Since $0\notin U(R)$, we have $\left\vert U(R)\right\vert <\left\vert
R\right\vert \,$\ and hence $\left\vert U^{k}\left(  R\right)  \right\vert
<\left\vert U^{k-1}\left(  R\right)  \right\vert $. Thus, $U^{k}\left(
R\right)  $ must eventually become a boolean ring and $U^{k+1}\left(
R\right)  $ is the trivial group. This mean the iterative structures of
$U^{i}\left(  R\right)  $ will reach the trivial group. These problems were
considered by some authors and arose a problem of determining all finite
commutative rings $R$ such that $U^{i}(R)$ is boolean or trivial group. Also
some considered the problem when $U^{i}(R)$ is a cyclic group for some rings
$R$ and values of $i$. El-Kassar and Chehade \cite{4} solved both problems
completely for $R=%
\mathbb{Z}
_{n}$ and $k=2$. Later, Kadri and El-Kassar in \cite{5}, considered the
problem for the case when $R=%
\mathbb{Z}
_{n}$ and $k=3$ and also provided a complete solution for these two problems.

In this paper, we examine the structure of the generalized group of units of $%
\mathbb{Z}
_{n}$. The structure is discussed by considering the two possible factors of
$n$ which are $2^{a}$ and $p_{i}^{\alpha_{i}},$ where $p_{i}$ is an odd prime
integer. Thus, we find a decomposition of $U^{k}\left(
\mathbb{Z}
_{n}\right)  $ as a direct product of cyclic groups for the general case of
any $k$. Also we examine the problem of having $U^{k}\left(  \mathbb{Z}%
_{n}\right)  $, a boolean ring and those that are trivial. We solve the
problem completely when $n=2^{\alpha}$, while the case when $n=p^{\alpha},p$
is an odd integer, is examined and some necessary conditions are given. Also
we give some properties of having $U^{k}\left(  \mathbb{Z}_{n}\right)  $ a
boolean or a trivial group. Eventually, we show that this decomposition
structure is directly related to the Pratt Tree primes, illustrated in an
example showing this relation.

\section{Some Preliminaries}

Let $R$ be the ring of integers modulo $n$, $\mathbb{Z}_{n}$. The
decomposition of the group of units of $\mathbb{Z}_{n}$, $U\left(
\mathbb{Z}_{n}\right)  $ can be found in \cite{6} stated in the following Lemma.

\begin{lemma}
\label{int 1}The group of units of $\mathbb{Z}_{n}$ when $n$ is a prime power
integer is given by
\end{lemma}

\begin{enumerate}
\item $U\left(  \mathbb{Z}_{2}\right)  \approx\left\{  0\right\}  $,

\item $U\left(  \mathbb{Z}_{2^{a}}\right)  \approx\mathbb{Z}_{2}%
\times\mathbb{Z}_{2^{a-2}}$ when $a\geq2$,

\item $U\left(  \mathbb{Z}_{p^{\alpha}}\right)  \approx\mathbb{Z}_{p-1}%
\times\mathbb{Z}_{p^{\alpha-1}}$ when $\alpha\geq1$.
\end{enumerate}

Thus the above isomorphism gives the structure of any group of units $U\left(
\mathbb{Z}_{n}\right)  $. If $n=2^{a}p_{1}^{\alpha_{1}}\cdots p_{i}%
^{\alpha_{i}}$ be the decomposition of $n$ into product of distinct prime
powers. Then
\[
\mathbb{Z}_{n}\cong\mathbb{Z}_{2^{a}}\oplus\mathbb{Z}_{p_{1}^{\alpha_{1}}%
}\oplus\cdots\oplus\mathbb{Z}_{p_{i}^{\alpha_{i}}}%
\]
and
\[
U\left(  \mathbb{Z}_{n}\right)  \thickapprox U\left(  \mathbb{Z}_{2^{a}%
}\right)  \times U\left(  \mathbb{Z}_{p_{1}^{\alpha_{1}}}\right)  \times
\cdots\times U\left(  \mathbb{Z}_{p_{i}^{\alpha_{i}}}\right)  .
\]
Moreover, we can conclude from this decomposition that $U\left(
\mathbb{Z}_{n}\right)  $ is a trivial group if and only if $n=1$ or $2.$
$U\left(  \mathbb{Z}_{n}\right)  $ is boolean ring for $a=2$ or $3$ and when
$U\left(  \mathbb{Z}_{p_{i}^{\alpha_{i}}}\right)  \approx\mathbb{Z}_{p_{i}%
-1}\times\mathbb{Z}_{p_{i}^{\alpha_{i}-1}}\approx\mathbb{Z}_{2}$ then
$\alpha_{i}=1$ and $p_{i}=3$. Then $U\left(  \mathbb{Z}_{n}\right)  $ is a
boolean, when $n=2^{2},2^{3},2^{2}\times3$ or $2^{3}\times3.$

El-Kassar and Chihade \cite{4}, introduced a generalization of the group of
units as the $k^{th}$ group of units of commutative ring with identity $R$
denoted as $U^{k}(R)$. The definition is based on the following theorem.

\begin{theorem}
If a group (G,*) is isomorphic to the additive group (R,+) of the ring
(R,+,.), then there is an operation $\oplus$ on G such that (G, *, $\oplus$)
is a ring isomorphic to (R,+,.).
\end{theorem}

Now, since $U(R)\thickapprox\mathbb{Z}_{n_{1}}\times\mathbb{Z}_{n_{2}}%
\times\cdots\times\mathbb{Z}_{n_{i}}$,\ we obtain that $U(R)$ is a ring
isomorphic to the direct sum of $\mathbb{Z}_{n}$'s. That is, $U(R)\cong%
\mathbb{Z}_{n_{1}}\oplus\mathbb{Z}_{n_{2}}\oplus\cdots\oplus\mathbb{Z}_{n_{i}}$.
Hence, the group of units of the ring $U(R)$ is called the second group of
units of $R$, written as $U^{2}(R)$. Continuing in applying the above steps,
we obtain the group of units $U^{k}(R)$ of the ring $U^{k-1}(R)$, which is
defined to be the generalized group of units of the commutative ring $R$ with
identity. Eventually, $U^{k}(R)$ shall be a commutative ring with identity.

The launching of this group opened several problems from studying the
structure of this group and determining all rings $R$ with a given
characteristic of $U^{k}\left(  R\right)  $. In particular, El-Kassar and
Chihade \cite{4} studied the decomposition of $U^{k}\left(  R\right)  $ in the
following theorem.

\begin{theorem}
\label{int 2}Let $k\geq0$. If $R\cong R_{1}\oplus R_{2}\oplus\cdots\oplus R_{r}$,
then
\[
U^{k}\left(  R\right)  \thickapprox U^{k}\left(  R_{1}\right)  \times
U^{k}\left(  R_{2}\right)  \times\cdots\times U^{k}\left(  R_{r}\right)
\]
and
\[
U^{k}\left(  R\right)  \cong U^{k}\left(  R_{1}\right)  \oplus U^{k}\left(
R_{2}\right)  \oplus\cdots\oplus U^{k}\left(  R_{r}\right)
\]
Note that the first is a group isomorphism and the second is a ring
isomorphism. So that at any step of this paper this isomorphism can represent
a group or a ring isomorphism. Also the zero group of units of $R$, $U^{0}%
(R)$, is the ring $R$ itself.
\end{theorem}

One of the most important classes is the ring of integers modulo $n$,
$\mathbb{Z}_{n}$. So if $n=p_{1}^{\alpha_{1}}p_{2}^{\alpha_{2}}\cdot\cdot\cdot
p_{i}^{\alpha_{i}}$ be the decomposition of $n$ into product of distinct prime
powers. Then
\[
U^{k}\left(  \mathbb{Z}_{n}\right)  \approx U^{k}\left(  \mathbb{Z}%
_{p_{1}^{\alpha_{1}}}\right)  \times U^{k}\left(  \mathbb{Z}_{p_{2}%
^{\alpha_{2}}}\right)  \times\cdots\times U^{k}\left(  \mathbb{Z}_{p_{i}%
^{\alpha_{i}}}\right)  .
\]

An application showing how these iterated groups are determined. Let
$R=\mathbb{Z}_{338}$. We have $\mathbb{Z}_{338}\cong\mathbb{Z}_{2}%
\oplus\mathbb{Z}_{13^{2}}$. Then the first group of units is $U(\mathbb{Z}%
_{338})$, which is isomorphic to $\mathbb{Z}_{12}\times\mathbb{Z}_{13}$. Now,
$U(\mathbb{Z}_{338})$ is a ring isomorphic to $\mathbb{Z}_{12}\oplus
\mathbb{Z}_{13}$. However, the group of units of $U(\mathbb{Z}_{338})$,
$U^{2}(\mathbb{Z}_{338}),$ is the second group of units of $\mathbb{Z}_{338}$.
$U^{2}(\mathbb{Z}_{338})$ is isomorphic to $\mathbb{Z}_{2}\times\mathbb{Z}%
_{2}\times\mathbb{Z}_{12}$, which is a ring isomorphic to $\mathbb{Z}%
_{2}\oplus\mathbb{Z}_{2}\oplus\mathbb{Z}_{12}$. Continuing in the same manner,
we obtain that $U^{3}(\mathbb{Z}_{338})$ is the third group of units of
$\mathbb{Z}_{338}$ isomorphic to $\mathbb{Z}_{2}\times\mathbb{Z}_{2}$. Also
$U^{4}(\mathbb{Z}_{338})$ is the $4^{th}$ group of units isomorphic to the
trivial ring $\mathbb{Z}_{1}=\left\{  0\right\}  $.

Also, for any ring $R$ with $\left\vert R\right\vert \,>1$. Since $0\notin
U(R)$, we have $\left\vert U(R)\right\vert <\left\vert R\right\vert \,$\ and
hence $\left\vert U^{k}\left(  R\right)  \right\vert <\left\vert
U^{k-1}\left(  R\right)  \right\vert $. Thus, $U^{k}\left(  R\right)  $ must
eventually become a boolean ring and $U^{k+1}\left(  R\right)  $ is the
trivial ring. In the above example $U^{3}(\mathbb{Z}_{338})$ is a boolean ring
and $U^{4}(\mathbb{Z}_{338})$ is the trivial ring.

El-Kassar and Chihade in \cite{4} solved the problem of determining all rings
$R$, such that $U^{k}\left(  R\right)  $ is trivial completely when
$R=\mathbb{Z}_{n}$ and $k=2$ summarized in the following theorem.

\begin{theorem}
$U^{2}\left(  \mathbb{Z}_{n}\right)  $ is trivial if and only if $n\ $divisor
of $24$.
\end{theorem}

Also Kadri and El-Kassar in \cite{5}, solved the problem for $U^{3}\left(
\mathbb{Z}_{n}\right)  $ given in the following theorem

\begin{theorem}
$U^{3}\left(  \mathbb{Z}_{n}\right)  $ is trivial if and only if $n\ $divisor
of $\ 131040.$
\end{theorem}

Moreover, they established a structure of $U^{3}\left(  \mathbb{Z}_{n}\right)
$ as
\[
U^{3}\left(  \mathbb{Z}_{2^{a}}\right)  \approx\left\{
\begin{tabular}
[c]{ll}%
$\left\{  0\right\}  $ & if $a<6$\\
$\mathbb{Z}_{2}\times\mathbb{Z}_{2^{a-6}}$ & if $a\geq6$%
\end{tabular}
\right.
\]
and when $p$ is an odd prime. Then
\[
U^{3}\left(  \mathbb{Z}_{p^{a}}\right)  \approx\left\{
\begin{tabular}
[c]{ll}%
$U^{2}(\mathbb{Z}_{p-1})$ & if $\alpha=1$\\
$U^{2}(\mathbb{Z}_{p-1})\times U(\mathbb{Z}_{p-1})$ & if $\alpha=2$\\
$U^{2}(\mathbb{Z}_{p-1})\times U(\mathbb{Z}_{p-1})\times\mathbb{Z}_{p-1}%
\times\mathbb{Z}_{p^{\alpha-3}}$ & if $\alpha\geq3$%
\end{tabular}
\right.
\]

\section{The Decomposition of $K^{th}$ Group Of Units Of $Z_{n}$}

In this section we determine the structure of the $k^{th}$ group of units of
$\mathbb{Z}_{n}$. First we consider the case when $n=$ $2^{\alpha}$ and then
the case when $n=p^{\alpha},$ where p is an odd prime.

\begin{lemma}
\label{a1}$U^{k}\left(  \mathbb{Z}_{2}\right)  \approx\left\{  0\right\}  $
for all $k\geq1$ and $U^{k}\left(  \mathbb{Z}_{4}\right)  \approx\left\{
0\right\}  $ for all $k\geq2$.
\end{lemma}

\begin{proof}
Let $k=1$. We have from Lemma \ref{int 1}, $U\left(  \mathbb{Z}_{2}\right)
\approx\left\{  0\right\}  $. Now, let $k>1$. We obtain $U^{k-1}\left(
U\left(  \mathbb{Z}_{2}\right)  \right)  \approx U^{k-1}(\left\{  0\right\}
)$ which gives that $U^{k}\left(  \mathbb{Z}_{2}\right)  \approx\left\{
0\right\}  $. Therefore, $U^{k}\left(  \mathbb{Z}_{2}\right)  \approx\left\{
0\right\}  $ for all $k\geq1$. Now, by Lemma \ref{int 1}, $U\left(
\mathbb{Z}_{4}\right)  \approx\mathbb{Z}_{2}$, and thus $U^{k-1}\left(
U\left(  \mathbb{Z}_{4}\right)  \right)  \approx U^{k-1}(\mathbb{Z}_{2}).$
However, from the previous result $U^{k-1}(\mathbb{Z}_{2})\approx\left\{
0\right\}  $ for $k-1\geq1,$ $k\geq2$. Therefore, $U^{k}\left(  \mathbb{Z}%
_{4}\right)  \approx\left\{  0\right\}  $ for all $k\geq2$.
\end{proof}

\begin{lemma}
\label{b3}Let $\alpha>2t\geq0$ and $k>t\geq0$. Then $U^{k}\left(
\mathbb{Z}_{2^{\alpha}}\right)  \approx U^{k-t}(\mathbb{Z}_{2^{\alpha-2t}}).$
\end{lemma}

\begin{proof}
Suppose that $\alpha>2$ and $k>1.$ By Lemma \ref{int 1}, we have $U\left(
\mathbb{Z}_{2^{\alpha}}\right)  \approx\mathbb{Z}_{2}\times\mathbb{Z}%
_{2^{\alpha-2}}$, then $U^{k-1}\left(  U\left(  \mathbb{Z}_{2^{\alpha}%
}\right)  \right)  \approx U^{k-1}\left(  \mathbb{Z}_{2}\mathbb{\times
Z}_{2^{\alpha-2}}\right)  $. However, $U^{k-1}\left(  \mathbb{Z}%
_{2}\mathbb{\times Z}_{2^{\alpha-2}}\right)  \approx U^{k-1}\left(
\mathbb{Z}_{2})\times U^{k-1}(\mathbb{Z}_{2^{\alpha-2}}\right)  $, and by
Lemma \ref{a1}, $U^{k-1}\left(  \mathbb{Z}_{2}\right)  \approx\left\{
0\right\}  $. Therefore, $U^{k}\left(  \mathbb{Z}_{2^{\alpha}}\right)  \approx
U^{k-1}(\mathbb{Z}_{2^{\alpha-2}})$.

Applying the above relation $t$ times, we obtain the result.
\end{proof}

\begin{lemma}
\label{b10}$U^{k}\left(  \mathbb{Z}_{2^{\alpha}}\right)  \approx
U^{k+t}\left(  \mathbb{Z}_{2^{\alpha+2t}}\right)  $\ for all nonzero natural
numbers $k$ and $\alpha$, where $t\geq0$.
\end{lemma}

\begin{proof}
Suppose that $k>0$ and $\alpha>0.$ Then by Lemma \ref{int 1}, $U\left(
\mathbb{Z}_{2^{\alpha+2}}\right)  \approx\mathbb{Z}_{2}\times\mathbb{Z}%
_{2^{\alpha}}$ and $U^{k+1}\left(  \mathbb{Z}_{2^{\alpha+2}}\right)  \approx
U^{k}\left(  \mathbb{Z}_{2}\right)  \times U^{k}(\mathbb{Z}_{2^{\alpha}})$.
But by Lemma \ref{a1}, $U^{k}\left(  \mathbb{Z}_{2}\right)  \approx\left\{
0\right\}  $. Hence, $U^{k+1}\left(  \mathbb{Z}_{2^{\alpha+2}}\right)  \approx
U^{k}\left(  \mathbb{Z}_{2^{\alpha}}\right)  .$

Applying the above relation $t$ times, we obtain the result.
\end{proof}

In the following theorem we give the decomposition of $U^{k}\left(
\mathbb{Z}_{2^{\alpha}}\right)  $ into a direct product of $\mathbb{Z}_{n}$'s.

\begin{theorem}
\label{b4}Let $k>0$ and $\alpha>0$. Then the decomposition of $U^{k}\left(
\mathbb{Z}_{2^{\alpha}}\right)  $ is given by
\end{theorem}

\begin{enumerate}
\item $U^{k}\left(  \mathbb{Z}_{2^{\alpha}}\right)  \approx\mathbb{Z}%
_{2}\times\mathbb{Z}_{2^{\alpha-2k}}$ if $\alpha>2k$,

\item $U^{k}\left(  \mathbb{Z}_{2^{\alpha}}\right)  \approx\mathbb{Z}_{2}$ if
$\alpha=2k$,

\item $U^{k}\left(  \mathbb{Z}_{2^{\alpha}}\right)  \approx\left\{  0\right\}
$ if $\alpha<2k$.
\end{enumerate}

Note that for $\alpha=2k+1$, $U^{k}\left(  \mathbb{Z}_{2^{2k+1}}\right)
\approx\mathbb{Z}_{2}\times\mathbb{Z}_{2}$ which is a boolean ring. Also (3)
can be written as: $U^{k}\left(  \mathbb{Z}_{2^{\alpha}}\right)
\approx\left\{  0\right\}  $ if and only $2^{\alpha}$ is a divisor of
$2^{2k-1}$.

\begin{proof}
Let $k>0$ and $\alpha>0.$

\begin{enumerate}
\item Suppose $\alpha>2k$ and $t=k-1$. Then $\alpha-2\left(  k-1\right)
=\alpha-2t>2$ and $k>t$. Now, from Lemma \ref{b3}, we have $U^{k}\left(
\mathbb{Z}_{2^{\alpha}}\right)  \approx U^{k-t}(\mathbb{Z}_{2^{\alpha-2t}})$.
Hence,
\begin{equation}
U^{k}\left(  \mathbb{Z}_{2^{\alpha}}\right)  \approx U^{k-\left(  k~-1\right)
}(\mathbb{Z}_{2^{\alpha-2\left(  k-1\right)  }})\text{ }=U(\mathbb{Z}%
_{2^{\alpha-2\left(  k-1\right)  }}) \label{eq 2}%
\end{equation}
and by Lemma \ref{int 1}, $U(\mathbb{Z}_{2^{\alpha-2\left(  k-1\right)  }%
})\approx\mathbb{Z}_{2}\times\mathbb{Z}_{2^{\alpha-2k}}$. Therefore,
$U^{k}\left(  \mathbb{Z}_{2^{\alpha}}\right)  \approx\mathbb{Z}_{2}%
\times\mathbb{Z}_{2^{\alpha-2k}}.$

\item Suppose that $\alpha=2k$ and $t=k-1$. Since $k>t$ and $\alpha
-2t=\alpha-2k+2=2>0$, we have from Lemma \ref{b3}, the same formula of
(\ref{eq 2}). But $U(\mathbb{Z}_{2^{\alpha-2\left(  k-1\right)  }%
})=U(\mathbb{Z}_{4})\approx\mathbb{Z}_{2}$. Therefore, $U^{k}\left(
\mathbb{Z}_{2^{\alpha}}\right)  \approx\mathbb{Z}_{2}$.

\item Suppose $\alpha<2k$. In the case $\alpha$ is odd, set $t=\frac{\alpha
-1}{2}$. Since $\alpha-2t=\alpha-2\left(  \frac{\alpha-1}{2}\right)  =1>0$ and
$k-t=k-\frac{\alpha-1}{2}=\frac{2k-\alpha+1}{2}>0$, Lemma \ref{b3} gives that
$U^{k}\left(  \mathbb{Z}_{2^{\alpha}}\right)  \approx U^{k-t}(\mathbb{Z}%
_{2^{1}})$ and $U^{k-t}(\mathbb{Z}_{2^{1}})$ is the trivial group. Now, the
case when $\alpha$ is even, set $t=\frac{\alpha-2}{2}$. Since $\alpha
-2t=\alpha-2\left(  \frac{\alpha-2}{2}\right)  =2>0$ and $k-t=k-\frac
{\alpha-2}{2}=\frac{2k-\alpha+2}{2}>1$. From Lemma \ref{b3}, $U^{k}\left(
\mathbb{Z}_{2^{\alpha}}\right)  \approx U^{k-t}(\mathbb{Z}_{2^{2}})$ which is
also the trivial group by Lemma \ref{a1}. Therefore, $U^{k}\left(
\mathbb{Z}_{2^{\alpha}}\right)  \approx\left\{  0\right\}  $ if $\alpha<2k$.
\end{enumerate}
\end{proof}

Next, we study the decomposition of $k^{th}$ group of units of the ring
$\mathbb{Z}_{p^{\alpha}}$,\ when $p$ is an odd prime.

\begin{lemma}
\label{c2}Let $p$ be an odd prime and let $k\geq1.$ Then $U^{k}\left(
\mathbb{Z}_{p}\right)  \approx U^{k-1}\left(  \mathbb{Z}_{p-1}\right)  $.
\end{lemma}

\begin{proof}
The proof is a direct consequence that $U(\mathbb{Z}_{p})\approx
\mathbb{Z}_{p-1}$.
\end{proof}

\begin{theorem}
\label{c3}Let $p$ be an odd prime and let $0\leq t<\alpha$ and $1\leq t\leq
k$. Then
\[
U^{k}\left(  \mathbb{Z}_{p^{\alpha}}\right)  \approx U^{k}\left(
\mathbb{Z}_{p}\right)  \times U^{k-1}\left(  \mathbb{Z}_{p}\right)
\times\cdots\times U^{k-t+1}\left(  \mathbb{Z}_{p}\right)  \times U^{k-t}\left(
\mathbb{Z}_{p^{\alpha-t}}\right)  .
\]

\end{theorem}

\begin{proof}
Suppose $t=0$, then $1<\alpha$ and $1\leq k$. We have from Lemma \ref{int 1},
$U\left(  \mathbb{Z}_{p^{\alpha}}\right)  \approx\mathbb{Z}_{p-1}%
\times\mathbb{Z}_{p^{\alpha-1}}$ and so $U^{k}\left(  \mathbb{Z}_{p^{\alpha}%
}\right)  \approx U^{k-1}\left(  \mathbb{Z}_{p-1}\right)  \times
U^{k-1}\left(  \mathbb{Z}_{p^{\alpha-1}}\right)  $. But by Lemma \ref{c2},
$U^{k-1}\left(  \mathbb{Z}_{p-1}\right)  \approx U^{k}\left(  \mathbb{Z}%
_{p}\right)  $. Hence,
\begin{equation}
U^{k}\left(  \mathbb{Z}_{p^{\alpha}}\right)  \approx U^{k}\left(
\mathbb{Z}_{p}\right)  \times U^{k-1}\left(  \mathbb{Z}_{p^{\alpha-1}}\right)
\label{eq 1}%
\end{equation}

Now, Suppose $t=2$, $2<\alpha$ and $2\leq k$. Then the isomorphism in
(\ref{eq 1}) can be written as
\[
U^{k-1}\left(  \mathbb{Z}_{p^{\alpha-1}}\right)  \text{ }\approx
U^{k-1}\left(  \mathbb{Z}_{p}\right)  \times U^{k-2}\left(  \mathbb{Z}%
_{p^{\alpha-2}}\right)  .
\]
by replacing $k$ and $\alpha$ by $k-1$ and $\alpha-1$ respectively. Hence,
\[
U^{k}\left(  \mathbb{Z}_{p^{\alpha}}\right)  \approx U^{k}\left(
\mathbb{Z}_{p}\right)  \times U^{k-1}\left(  \mathbb{Z}_{p}\right)  \times
U^{k-2}\left(  \mathbb{Z}_{p^{\alpha-2}}\right)  .
\]
Continuing in the same manner. When $t<\alpha$ and $t\leq k$, we conclude
that, $U^{k-(t-1)}\left(  \mathbb{Z}_{p^{\alpha-(t-1)}}\right)  \approx
U^{k-(t-1)}\left(  \mathbb{Z}_{p}\right)  \times U^{k-t}\left(  \mathbb{Z}%
_{p^{\alpha-t}}\right)  $. Therefore,
\[
U^{k}\left(  \mathbb{Z}_{p^{\alpha}}\right)  \approx U^{k}\left(
\mathbb{Z}_{p}\right)  \times U^{k-1}\left(  \mathbb{Z}_{p}\right)
\times\cdots\times U^{k-t+1}\left(  \mathbb{Z}_{p}\right)  \times U^{k-t}\left(
\mathbb{Z}_{p^{\alpha-t}}\right)  \text{. }%
\]

\end{proof}

\begin{example}
Applying the above theorem we obtain the following. Let $p=47,$ $k=8$,
$\alpha=6$ and $t=5.$ Then
\[
U^{8}\left(  \mathbb{Z}_{47^{6}}\right)  \approx U^{8}\left(  \mathbb{Z}%
_{47}\right)  \times U^{7}\left(  \mathbb{Z}_{47}\right)  \times\cdots\times
U^{4}\left(  \mathbb{Z}_{47}\right)  \times U^{3}\left(  \mathbb{Z}%
_{47}\right)  \text{.}%
\]
But $U\left(  \mathbb{Z}_{47}\right)  \approx\mathbb{Z}_{46}\approx
\mathbb{Z}_{2}\times\mathbb{Z}_{23}$, which implies that $U^{2}\left(
\mathbb{Z}_{47}\right)  \approx U(\mathbb{Z}_{23})\approx\mathbb{Z}%
_{22}\approx\mathbb{Z}_{2}\times\mathbb{Z}_{11}\,$and so $U^{3}\left(
\mathbb{Z}_{47}\right)  \approx U(\mathbb{Z}_{11})\approx$ $\mathbb{Z}%
_{10}\approx\mathbb{Z}_{2}\times\mathbb{Z}_{5}$, $U^{4}\left(  \mathbb{Z}%
_{47}\right)  \approx U(\mathbb{Z}_{5})\approx\mathbb{Z}_{4}$ $\,$ and
$U^{5}\left(  \mathbb{Z}_{47}\right)  \approx U(\mathbb{Z}_{4})\approx
\mathbb{Z}_{2}$ and hence, $U^{6}\left(  \mathbb{Z}_{47}\right)  \approx
U^{7}\left(  \mathbb{Z}_{47}\right)  \approx U^{8}\left(  \mathbb{Z}%
_{47}\right)  \approx\left\{  0\right\}  $. Therefore,
\[
U^{8}\left(  \mathbb{Z}_{47^{6}}\right)  \approx U^{5}\left(  \mathbb{Z}%
_{47}\right)  \times U^{4}\left(  \mathbb{Z}_{47}\right)  \times U^{3}\left(
\mathbb{Z}_{47}\right)  \approx\mathbb{Z}_{2}\times\mathbb{Z}_{4}%
\,\times\mathbb{Z}_{2}\times\mathbb{Z}_{5}.\text{ }%
\]

\end{example}

\begin{theorem}
\label{th1}Let $p$ be an odd prime and let $\alpha>0$ and $k>0$. Then
\end{theorem}

\begin{enumerate}
\item $U^{k}\left(  \mathbb{Z}_{p^{\alpha}}\right)  \approx U^{k}\left(
\mathbb{Z}_{p}\right)  \times U^{k-1}\left(  \mathbb{Z}_{p}\right)
\times\cdots\times U^{k-\alpha+1}\left(  \mathbb{Z}_{p}\right)  $, when
$\alpha<k,$

\item $U^{k}\left(  \mathbb{Z}_{p^{\alpha}}\right)  \approx U^{k}\left(
\mathbb{Z}_{p}\right)  \times U^{k-1}\left(  \mathbb{Z}_{p}\right)
\times\cdots\times U^{2}\left(  \mathbb{Z}_{p}\right)  \times U\left(
\mathbb{Z}_{p}\right)  $, when $\alpha=k,$

\item $U^{k}\left(  \mathbb{Z}_{p^{\alpha}}\right)  \approx U^{k}\left(
\mathbb{Z}_{p}\right)  \times U^{k-1}\left(  \mathbb{Z}_{p}\right)
\times\cdots\times U\left(  \mathbb{Z}_{p}\right)  \times\mathbb{Z}_{p^{\alpha
-k}}$, when $\alpha>k.$
\end{enumerate}

\begin{proof}
Let $p$ be an odd prime and let $\alpha>0$ and $k>0$.

\begin{enumerate}
\item Let $\alpha\leq k$ and $t=\alpha-1$. Then $t<\alpha$ and $t<k$ and by
Theorem \ref{c3},
\[
U^{k}\left(  \mathbb{Z}_{p^{\alpha}}\right)  \approx U^{k}\left(
\mathbb{Z}_{p}\right)  \times U^{k-1}\left(  \mathbb{Z}_{p}\right)
\times\cdots\times U^{k-t+1}\left(  \mathbb{Z}_{p}\right)  \times U^{k-t}\left(
\mathbb{Z}_{p^{\alpha-t}}\right)  .\text{ }%
\]
Thus,%
\begin{equation}
U^{k}\left(  \mathbb{Z}_{p^{\alpha}}\right)  \approx U^{k}\left(
\mathbb{Z}_{p}\right)  \times U^{k-1}\left(  \mathbb{Z}_{p}\right)
\times\cdots\times U^{k-\alpha+2}\left(  \mathbb{Z}_{p}\right)  \times
U^{k-\alpha+1}\left(  \mathbb{Z}_{p}\right)  . \label{eq3}%
\end{equation}

\item The proof is obtained by replacing $\alpha=k$ in the isomorphism
(\ref{eq3}).

\item Let $\alpha>k$ and $t=k-1.$ Then $t<\alpha$ and $t<k.$ From Theorem
\ref{c3},
\begin{align*}
U^{k}\left(  \mathbb{Z}_{p^{\alpha}}\right)   &  \approx U^{k}\left(
\mathbb{Z}_{p}\right)  \times U^{k-1}\left(  \mathbb{Z}_{p}\right)
\times\cdots\times U^{k-t+1}\left(  \mathbb{Z}_{p}\right)  \times U^{k-t}\left(
\mathbb{Z}_{p^{\alpha-t}}\right)  \text{ }\\
&  \approx U^{k}\left(  \mathbb{Z}_{p}\right)  \times U^{k-1}\left(
\mathbb{Z}_{p}\right)  \times\cdots\times U^{k-\left(  k-1\right)  +1}\left(
\mathbb{Z}_{p}\right)  \times U^{k-\left(  k-1\right)  }\left(  \mathbb{Z}%
_{p^{\alpha-\left(  k-1\right)  }}\right) \\
&  \approx U^{k}\left(  \mathbb{Z}_{p}\right)  \times U^{k-1}\left(
\mathbb{Z}_{p}\right)  \times\cdots\times U^{2}\left(  \mathbb{Z}_{p}\right)
\times U\left(  \mathbb{Z}_{p^{\alpha-k+1}}\right)  .
\end{align*}
Now, since $\alpha-k+1>0$, Lemma \ref{int 1} gives that $U\left(
\mathbb{Z}_{p^{\alpha-k+1}}\right)  \approx\mathbb{Z}_{p-1}\times
\mathbb{Z}_{p^{\alpha-k}}$. Therefore,
\[
U^{k}\left(  \mathbb{Z}_{p^{\alpha}}\right)  \approx U^{k}\left(
\mathbb{Z}_{p}\right)  \times U^{k-1}\left(  \mathbb{Z}_{p}\right)
\times\cdots\times U^{2}\left(  \mathbb{Z}_{p}\right)  \times\mathbb{Z}%
_{p-1}\times\mathbb{Z}_{p^{\alpha-k}}.
\]

\end{enumerate}
\end{proof}

The above theorem gives the decomposition of $U^{k}\left(  \mathbb{Z}%
_{p^{\alpha}}\right)  $ into a direct product of $U^{i}\left(  \mathbb{Z}%
_{p}\right)  $ and $\mathbb{Z}_{p^{j}}$. So by finding the decomposition of
$U^{i}\left(  \mathbb{Z}_{p}\right)  $, for a given odd prime $p$, the
decomposition $U^{k}\left(  \mathbb{Z}_{p^{\alpha}}\right)  $ is established.

Next, we give some application of decompositions of $U^{k}\left(
\mathbb{Z}_{p^{\alpha}}\right)  $ in the case $p=3$.

\begin{corollary}
\label{b7}Let $n=3^{\alpha}$. Then the decomposition of the $k^{th}$ group of
units of $\mathbb{Z}_{n}$ is given by
\[
U^{k}\left(  \mathbb{Z}_{3^{\alpha}}\right)  \approx\left\{
\begin{tabular}
[c]{ll}%
$\mathbb{Z}_{2}\times\mathbb{Z}_{3^{\alpha-k}\text{ }}$ & if $\alpha>k$\\
$\mathbb{Z}_{2}$ & if $\alpha=k$\\
$\left\{  0\right\}  $ & if $\alpha<k$%
\end{tabular}
\right.
\]

\end{corollary}

\begin{proof}
We have $U^{i}\left(  \mathbb{Z}_{3}\right)  =U^{i-1}\left(  U\left(
\mathbb{Z}_{3}\right)  \right)  \approx U^{i-1}(\mathbb{Z}_{2})$. But
$U^{i-1}(\mathbb{Z}_{2})=\mathbb{Z}_{2}$ if $i=1$ and $U^{i-1}(\mathbb{Z}%
_{2})\approx\left\{  0\right\}  $ for $i>1$. Hence,
\[
U^{i}\left(  \mathbb{Z}_{3}\right)  \approx\left\{
\begin{array}
[c]{c}%
\mathbb{Z}_{2}~\text{if }i=1\\
\left\{  0\right\}  ~\text{if }i>1
\end{array}
\right.  \text{ }%
\]

By applying Theorem \ref{th1} for $p=3\,,$ we obtain that when $\alpha<k$,
\[
U^{k}\left(  \mathbb{Z}_{3^{\alpha}}\right)  \approx U^{k}\left(
\mathbb{Z}_{3}\right)  \times U^{k-1}\left(  \mathbb{Z}_{3}\right)
\times\cdots\times U^{k-\alpha+1}\left(  \mathbb{Z}_{3}\right)  .
\]
But $k-j+1>1$ for $j=1,2,\cdots,\alpha.$ and so $U^{k}\left(  \mathbb{Z}%
_{3}\right)  \approx U^{k-1}\left(  \mathbb{Z}_{3}\right)  \approx\cdots\approx
U^{k-\alpha+1}\left(  \mathbb{Z}_{3}\right)  \approx\left\{  0\right\}  $.
Therefore, $U^{k}\left(  \mathbb{Z}_{3^{\alpha}}\right)  \approx\left\{
0\right\}  $.

Now, if $\alpha=k$,%
\[
U^{k}\left(  \mathbb{Z}_{3^{k}}\right)  \approx U^{k}\left(  \mathbb{Z}%
_{3}\right)  \times U^{k-1}\left(  \mathbb{Z}_{3}\right)  \times\cdots\times
U^{2}\left(  \mathbb{Z}_{3}\right)  \times U\left(  \mathbb{Z}_{3}\right)  .
\]
But $U^{k}\left(  \mathbb{Z}_{3}\right)  \approx U^{k-1}\left(  \mathbb{Z}%
_{3}\right)  \approx U^{2}\left(  \mathbb{Z}_{3}\right)  \approx\left\{
0\right\}  $. Hence, $U^{k}\left(  \mathbb{Z}_{3^{k}}\right)  \approx U\left(
\mathbb{Z}_{3}\right)  \approx\mathbb{Z}_{2}$.

If $\alpha>k,$ then
\begin{align*}
U^{k}\left(  \mathbb{Z}_{3^{\alpha}}\right)   &  \approx U^{k}\left(
\mathbb{Z}_{3}\right)  \times U^{k-1}\left(  \mathbb{Z}_{3}\right)
\times\cdots\times U^{2}\left(  \mathbb{Z}_{3}\right)  \times U\left(
\mathbb{Z}_{3}\right)  \times\mathbb{Z}_{3^{\alpha-k}}\\
&  \approx\mathbb{Z}_{2}\times\mathbb{Z}_{3^{\alpha-k}}.
\end{align*}

\end{proof}

The following corollary is a direct conclusion done by combining Theorem
\ref{th1} and Lemma \ref{c2}.

\begin{corollary}
\label{th2}Let $p$ be an odd prime and let $\alpha>0$ and $k>0$. Then
\end{corollary}

\begin{enumerate}
\item $U^{k}\left(  \mathbb{Z}_{p^{\alpha}}\right)  \approx U^{k-1}\left(
\mathbb{Z}_{p-1}\right)  \times U^{k-2}\left(  \mathbb{Z}_{p-1}\right)
\times\cdots\times U^{k-\alpha}\left(  \mathbb{Z}_{p-1}\right)  $, when
$\alpha<k,$

\item $U^{k}\left(  \mathbb{Z}_{p^{k}}\right)  \approx U^{k-1}\left(
\mathbb{Z} _{p-1}\right)  \times U^{k-2}\left(  \mathbb{Z}_{p-2}\right)
\times\cdots\times U\left(  \mathbb{Z}_{p-1}\right)  \times\mathbb{Z}_{p-1}$,
when $\alpha=k,$

\item $U^{k}\left(  \mathbb{Z}_{p^{\alpha}}\right)  \approx U^{k-1}\left(
\mathbb{Z}_{p-1}\right)  \times U^{k-2}\left(  \mathbb{Z}_{p-1}\right)
\times\cdots\times\mathbb{Z}_{p^{\alpha-k}(p-1)}$, when $\alpha>k.$
\end{enumerate}

The next corollaries refer to Corollary \ref{th2} in determining the structure
of $U^{k}\left(  \mathbb{Z}_{p^{\alpha}}\right)  $ by knowing the structure of
$U^{i}\left(  \mathbb{Z}_{p-1}\right)  $ where $i<k$. We apply this on
$U^{k}\left(  \mathbb{Z}_{5^{\alpha}}\right)  $ and $U^{k}\left(
\mathbb{Z}_{7^{\alpha}}\right)  $.

\begin{corollary}
Let $n=5^{\alpha}$. Then the decomposition of the $k^{th}$ group of units of
$\mathbb{Z}_{n}$ is given
\[
U^{k}\left(  \mathbb{Z}_{5^{\alpha}}\right)  \approx\left\{
\begin{tabular}
[c]{ll}%
$\mathbb{Z}_{2}\times\mathbb{Z}_{4\text{ }}\times\mathbb{Z}_{5^{\alpha
-k}\text{ }}$ & if $\alpha>k.$\\
$\mathbb{Z}_{2}\times\mathbb{Z}_{4}$ & if $\alpha=k$\\
$\mathbb{Z}_{2}$ & if $\alpha=k-1$\\
$\left\{  0\right\}  $ & if $\alpha<k-1$%
\end{tabular}
\right.
\]

\end{corollary}

\begin{proof}
Using Theorem \ref{b4}, and setting $\alpha=2$ the case $\alpha>2k$ is
rejected, so we are left with
\begin{align*}
U^{k}\left(  \mathbb{Z}_{4}\right)   &  \approx\left\{
\begin{tabular}
[c]{ll}%
$\mathbb{Z}_{2}$ & if $k=1$\\
$\left\{  0\right\}  $ & if $k>1$%
\end{tabular}
\right. \\
U^{k-i}\left(  \mathbb{Z}_{4}\right)   &  \approx\left\{
\begin{tabular}
[c]{ll}%
$\mathbb{Z}_{2}$ & if $i=k-1$\\
$\left\{  0\right\}  $ & if $i<k-1$%
\end{tabular}
\right.  ,\text{ }i=1,2,\ldots,k-1
\end{align*}

By applying Corollary \ref{th2} we get
\[
U^{k}\left(  \mathbb{Z}_{5^{\alpha}}\right)  \approx\left\{
\begin{tabular}
[c]{ll}%
$U^{k-1}\left(  \mathbb{Z}_{4}\right)  \times U^{k-2}\left(  \mathbb{Z}%
_{4}\right)  \times\cdots\times U^{k-\alpha}\left(  \mathbb{Z}_{4}\right)  $ &
if $\alpha<k$\\
$U^{k-1}\left(  \mathbb{Z}_{4}\right)  \times U^{k-2}\left(  \mathbb{Z}%
_{4}\right)  \times\cdots\times U\left(  \mathbb{Z}_{4}\right)  \times
\mathbb{Z}_{4}$ & if $\alpha=k$\\
$U^{k-1}\left(  \mathbb{Z}_{4}\right)  \times U^{k-2}\left(  \mathbb{Z}%
_{4}\right)  \times\cdots\times U\left(  \mathbb{Z}_{4}\right)  \times
\mathbb{Z}_{4}\times\mathbb{Z}_{5^{\alpha-k}}$ & if $\alpha>k$%
\end{tabular}
\right.
\]

for $\alpha<k$, if $\alpha=k-1$ $U^{k-\alpha}\left(  \mathbb{Z}_{4}\right)
\approx\mathbb{Z}_{2}$ and $U^{k-i}\left(  \mathbb{Z}_{4}\right)
\approx\left\{  0\right\}  $ for $i<k-1,$ thus $U^{k}\left(  \mathbb{Z}%
_{5^{\alpha}}\right)  \approx\mathbb{Z}_{2}$. and if $\alpha<k-1,$ we have
$U^{k-i}\left(  \mathbb{Z}_{4}\right)  \approx\left\{  0\right\}  $ for
$i=1,2,\ldots,\alpha$. Thus $U^{k}\left(  \mathbb{Z}_{5^{\alpha}}\right)
\approx\mathbb{Z}_{2}$.

For the second case $\alpha=k,$ all the summands are trivial except $U\left(
\mathbb{Z}_{4}\right)  \approx\mathbb{Z}_{2}$. Then $U^{k}\left(
\mathbb{Z}_{5^{\alpha}}\right)  \approx\mathbb{Z}_{2}\times\mathbb{Z}_{4}$.
Consequently to the case $\alpha=k$, we can conclude directly the case
$\alpha>k$, that is $U^{k}\left(  \mathbb{Z}_{5^{\alpha}}\right)
\approx\mathbb{Z}_{2}\times\mathbb{Z}_{4}\times\mathbb{Z}_{5^{\alpha-k}}$.
\end{proof}

\begin{corollary}
Let $n=7^{\alpha}$. Then the decomposition of the $k^{th}$ group of units of
$\mathbb{Z}_{n}$ is given
\[
U^{k}\left(  \mathbb{Z}_{7^{\alpha}}\right)  \approx\left\{
\begin{tabular}
[c]{ll}%
$\mathbb{Z}_{2}\times\mathbb{Z}_{6}\times\mathbb{Z}_{7^{\alpha-k}}$ & if
$\alpha>k.$\\
$\mathbb{Z}_{2}\times\mathbb{Z}_{6}$ & if $\alpha=k$\\
$\mathbb{Z}_{2}$ & if $\alpha=k-1$\\
$\left\{  0\right\}  $ & if $\alpha<k-1$%
\end{tabular}
\right.
\]

\end{corollary}

\begin{proof}
From Corollary \ref{th2} we relate the decomposition of $U^{k}\left(
\mathbb{Z}_{p^{\alpha}}\right)  $ to $U^{i}\left(  \mathbb{Z}_{p-1}\right)
\,$for $i<k$ then for $p=7$ we need to find $U^{i}\left(  \mathbb{Z}%
_{6}\right)  \,$for $i<k.$ we have $U^{i}\left(  \mathbb{Z}_{6}\right)
\approx U^{i}\left(  \mathbb{Z}_{2}\right)  \times U^{i}\left(  \mathbb{Z}%
_{3}\right)  $. However from Theorem \ref{b4} $U^{i}\left(  \mathbb{Z}%
_{2}\right)  \approx\left\{  0\right\}  $ and from Corollary \ref{b7}
\[
U^{i}\left(  \mathbb{Z}_{3}\right)  \approx\left\{
\begin{array}
[c]{c}%
\mathbb{Z}_{2}~\text{if }i=1\\
\left\{  0\right\}  ~\text{if }i>1
\end{array}
\right.  \text{ }%
\]
then
\[
U^{i}\left(  \mathbb{Z}_{6}\right)  \approx\left\{
\begin{array}
[c]{c}%
\mathbb{Z}_{2}~\text{if }i=1\\
\left\{  0\right\}  ~\text{if }i>1
\end{array}
\right.  \text{ }%
\]

By applying Corollary \ref{th2} we get

for $\alpha>k$,
\begin{align*}
U^{k}\left(  \mathbb{Z}_{7^{\alpha}}\right)   &  \approx U^{k-1}\left(
\mathbb{Z}_{6}\right)  \times U^{k-2}\left(  \mathbb{Z}_{6}\right)
\times\cdots\times U\left(  \mathbb{Z}_{6}\right)  \times\mathbb{Z}_{6}%
\times\mathbb{Z}_{7^{\alpha-k}}\\
&  \approx\mathbb{Z}_{2}\times\mathbb{Z}_{6}\times\mathbb{Z}_{7^{\alpha-k}}%
\end{align*}
for $\alpha=k$,
\begin{align*}
U^{k}\left(  \mathbb{Z}_{7^{\alpha}}\right)   &  \approx U^{k-1}\left(
\mathbb{Z}_{6}\right)  \times U^{k-2}\left(  \mathbb{Z}_{6}\right)
\times\cdots\times U\left(  \mathbb{Z}_{6}\right)  \times\mathbb{Z}_{6}\\
&  \approx\mathbb{Z}_{2}\times\mathbb{Z}_{6}%
\end{align*}

for $\alpha<k$, if $\alpha=k-1$
\begin{align*}
U^{k}\left(  \mathbb{Z}_{7^{\alpha}}\right)   &  \approx U^{k-1}\left(
\mathbb{Z}_{6}\right)  \times U^{k-2}\left(  \mathbb{Z}_{6}\right)
\times\cdots\times U\left(  \mathbb{Z}_{6}\right) \\
&  \approx\mathbb{Z}_{2}%
\end{align*}
and if $\alpha=k-2$
\begin{align*}
U^{k}\left(  \mathbb{Z}_{7^{\alpha}}\right)   &  \approx U^{k-1}\left(
\mathbb{Z}_{6}\right)  \times U^{k-2}\left(  \mathbb{Z}_{6}\right)
\times\cdots\times U^{2}\left(  \mathbb{Z}_{6}\right) \\
&  \approx\left\{  0\right\}  .
\end{align*}

and thus $U^{k}\left(  \mathbb{Z}_{7^{\alpha}}\right)  \approx\left\{
0\right\}  $ for $\alpha<k-2$. Therefore, we obtain the required.
\end{proof}

We end this section by noting that for higher prime integers the decomposition
is more complicated. But we noticed that the decomposition of $U^{k}\left(
\mathbb{Z}_{3^{\alpha}}\right)  $ and $U^{k}\left(  \mathbb{Z}_{5^{\alpha}%
}\right)  $ were obtained knowing the decomposition of $U^{k}\left(
\mathbb{Z}_{2}\right)  $. Also the decomposition of $U^{k}\left(
\mathbb{Z}_{7^{\alpha}}\right)  $ is obtained from the decomposition of
$U^{k}\left(  \mathbb{Z}_{2}\right)  $ and $U^{k}\left(  \mathbb{Z}%
_{3}\right)  $ and so on. We may conclude that each decomposition of
$U^{k}\left(  \mathbb{Z}_{p^{\alpha}}\right)  $ has a Tree of decompositions
of $U^{k}\left(  \mathbb{Z}_{p_{j}}\right)  $ for a given sequence of primes
$p_{i}$. This problem is discussed in Section 5.

\section{Boolean and Trivial $U^{k}\left(  Z_{n}\right)  $}

The previous section opened the importance in examining the rings that have
$k^{th}$ group of units, $U^{k}\left(  \mathbb{Z}_{n}\right)  $, a boolean and
those that are trivial. In this section, we study these two problems. First,
we consider the case $n=2^{\alpha}$, then when $n=p^{\alpha},$\ where $p$ is a
odd prime. We solve the problem completely when $n=2^{\alpha}$, while the case
when $n=p^{\alpha}$ is examined and some necessary conditions are given. We
end this section by concluding some properties of having $U^{k}\left(
\mathbb{Z}_{n}\right)  $ a boolean or a trivial group.

In the following theorem our two major problems are solved in the case
$n=2^{\alpha}$ and $n=3^{\alpha}$.

\begin{theorem}
\label{r4}Let $\alpha\geq1$ and $k\geq1$. Then

\begin{enumerate}
\item $U^{k}\left(  \mathbb{Z}_{2^{\alpha}}\right)  $ is a boolean ring if and
only if $\alpha=2k$ or $\alpha=2k+1\,$and is trivial if and only if
$\alpha<2k.$

\item $U^{k}\left(  \mathbb{Z}_{3^{\alpha}}\right)  $ is boolean ring if and
only if $\alpha=k$ and is trivial if and only if $\alpha<k$.
\end{enumerate}
\end{theorem}

\begin{proof}
The proof is a direct consequence from Theorem \ref{b4} and Corollary \ref{b7}.
\end{proof}

Next, we consider the case when $p$ is an odd prime integer and since in
Theorem \ref{r4} the special case $p=3$ is solved so we may consider the cases
when $p$ is an odd prime integer different than $3$.

\begin{lemma}
\label{d1}Let $p$ be an odd prime different from $3$. If $U^{k}\left(
\mathbb{Z}_{p^{\alpha}}\right)  $ is boolean ring, then $\alpha<k$.
\end{lemma}

\begin{proof}
Let $p$ be an odd prime different from $3$ and suppose that $U^{k}\left(
\mathbb{Z}_{p^{\alpha}}\right)  $ is a boolean ring. Assume for contradiction
that $\alpha\geq k$. If $\alpha=k,$ then by Theorem \ref{th1} , we have
\[
U^{k}\left(  \mathbb{Z}_{p^{k}}\right)  \approx U^{k}\left(  \mathbb{Z}
_{p}\right)  \times U^{k-1}\left(  \mathbb{Z}_{p}\right)  \times\cdots\times
U^{2}\left(  \mathbb{Z}_{p}\right)  \times U\left(  \mathbb{Z}_{p}\right)  .
\]
Hence, $U^{k}\left(  \mathbb{Z}_{p^{k}}\right)  $ is boolean if and only if
$U\left(  \mathbb{Z}_{p}\right)  $ is a boolean ring implies that $p=3$ a contradiction.

Now, suppose that $\alpha>k.$ By Theorem \ref{th1}, we have
\[
U^{k}\left(  \mathbb{Z}_{p^{\alpha}}\right)  \approx U^{k}\left(  \mathbb{Z}
_{p}\right)  \times U^{k-1}\left(  \mathbb{Z}_{p}\right)  \times\cdots\times
U\left(  \mathbb{Z}_{p}\right)  \times\mathbb{Z}_{p^{\alpha-k}}.
\]
Hence, $\mathbb{Z}_{p^{\alpha-k}}$ is boolean or trivial, a contradiction.
Therefore, $\alpha<k.$
\end{proof}

\begin{lemma}
\label{d11}Let $p$ be an odd prime$.$ If $U^{k}\left(  \mathbb{Z}_{p^{\alpha}%
}\right)  $ is trivial, then $\alpha<k.$
\end{lemma}

\begin{proof}
Let $p$ be an odd prime and suppose that $U^{k}\left(  \mathbb{Z}_{p^{\alpha}%
}\right)  $ is trivial. Suppose that $\alpha=k$, then by Theorem \ref{th1}, we
have
\[
U^{k}\left(  \mathbb{Z}_{p^{k}}\right)  \approx U^{k}\left(  \mathbb{Z}
_{p}\right)  \times U^{k-1}\left(  \mathbb{Z}_{p}\right)  \times\cdots\times
U^{2}\left(  \mathbb{Z}_{p}\right)  \times U\left(  \mathbb{Z}_{p}\right)
\text{ .}%
\]
Hence, $U\left(  \mathbb{Z}_{p}\right)  $ is trivial, since $U^{k}\left(
\mathbb{Z}_{p^{k}}\right)  $ is trivial if and only if
\[
U^{k}\left(  \mathbb{Z}_{p}\right)  \approx U^{k-1}\left(  \mathbb{Z}
_{p}\right)  \approx\cdots\approx U\left(  \mathbb{Z}_{p}\right)
\approx\left\{  0\right\}  .
\]
But $U\left(  \mathbb{Z}_{p}\right)  $ is trivial implies that $p=1,2,$ a
contradiction. Therefore $\alpha\neq k.$

Now, suppose that $\alpha>k.$ By Theorem \ref{th1}, we have
\[
U^{k}\left(  \mathbb{Z}_{p^{\alpha}}\right)  \approx U^{k}\left(  \mathbb{Z}
_{p}\right)  \times U^{k-1}\left(  \mathbb{Z}_{p}\right)  \times\cdots\times
U\left(  \mathbb{Z}_{p}\right)  \times\mathbb{Z}_{p^{\alpha-k}}.
\]
But $U^{k}\left(  \mathbb{Z}_{p^{\alpha}}\right)  $ is trivial if and only if
\[
U^{k-1}\left(  \mathbb{Z}_{p-1}\right)  \approx U^{k-2}\left(  \mathbb{Z}
_{p-1}\right)  \approx\cdots\approx U\left(  \mathbb{Z}_{p}\right)
\approx\mathbb{Z}_{p^{\alpha-k}}\approx\left\{  0\right\}  ;
\]
a contradiction, as $\mathbb{Z}_{p^{\alpha-k}}$ is never trivial. Therefore,
$\alpha<k.$
\end{proof}

From the previous two Lemmas we conclude one of necessary condition to have
$U^{k}\left(  \mathbb{Z}_{p^{\alpha}}\right)  $ a boolean ring or a trivial
one which is $\alpha<k$. Next, we find the sufficient condition to obtain
these rings.

\begin{theorem}
\label{th22}Let $p$ be an odd prime different from 3. Then $U^{k}\left(
\mathbb{Z}_{p^{\alpha}}\right)  $ is boolean ring if and only if $\alpha<k$
and $U^{k-\alpha+1}\left(  \mathbb{Z}_{p}\right)  $ is a boolean ring.
Moreover,%
\[
U^{k}\left(  \mathbb{Z}_{p^{\alpha}}\right)  \approx U^{k-\alpha+1}\left(
\mathbb{Z}_{p}\right)  .
\]

\end{theorem}

\begin{proof}
Let $p$ be an odd prime different from 3 and let $U^{k}\left(  \mathbb{Z}
_{p^{\alpha}}\right)  $ be a boolean ring. Then from Lemma \ref{d1}, we obtain
that $\alpha<k.$ Also by Theorem \ref{th1},
\[
U^{k}\left(  \mathbb{Z}_{p^{\alpha}}\right)  \approx U^{k}\left(  \mathbb{Z}
_{p}\right)  \times U^{k-1}\left(  \mathbb{Z}_{p}\right)  \times\cdots\times
U^{k-\alpha+1}\left(  \mathbb{Z}_{p}\right)  .
\]
Then, $U^{k-\alpha+1}\left(  \mathbb{Z}_{p}\right)  $ is a boolean ring also
\[
U^{k-\alpha+2}\left(  \mathbb{Z}_{p}\right)  \approx U^{k-\alpha+3}\left(
\mathbb{Z}_{p}\right)  \approx\cdots\approx U^{k-\alpha+\alpha}\left(
\mathbb{Z}_{p}\right)  \approx\left\{  0\right\}  .
\]
Therefore, $U^{k}\left(  \mathbb{Z}_{p^{\alpha}}\right)  \approx
U^{k-\alpha+1}\left(  \mathbb{Z}_{p}\right)  .$

Conversely, let $\alpha<k$ and $U^{k-\alpha+1}\left(  \mathbb{Z}_{p}\right)  $
is a boolean ring. Then \newline$U^{s}(U^{k-\alpha+1}\left(  \mathbb{Z}%
_{p}\right)  )$ $\approx\left\{  0\right\}  ,$ when $s=1,2\cdots,\alpha-1.$ That
is,
\[
U^{k}\left(  \mathbb{Z}_{p}\right)  \approx U^{k-1}\left(  \mathbb{Z}
_{p}\right)  \approx\cdots\approx U^{k-\alpha+2}\left(  \mathbb{Z}
_{p}\right)  \approx\left\{  0\right\}  .
\]
Therefore, $U^{k}\left(  \mathbb{Z}_{p^{\alpha}}\right)  \approx
U^{k-\alpha+1}\left(  \mathbb{Z}_{p}\right)  $, which is a boolean ring.
\end{proof}

\begin{theorem}
\label{d2}Let $p$ be an odd prime. Then $U^{k}\left(  \mathbb{Z}_{p^{\alpha}%
}\right)  $ is trivial if and only if $\alpha<k$ and $U^{k-\alpha+1}\left(
\mathbb{Z}_{p}\right)  \approx\left\{  0\right\}  $.
\end{theorem}

\begin{proof}
Let $p$ be an odd prime and suppose that $U^{k}\left(  \mathbb{Z}_{p^{\alpha}%
}\right)  $ is trivial. Then from Lemma \ref{d11}, we obtain that $\alpha<k. $
Also by Theorem \ref{th1},
\[
U^{k}\left(  \mathbb{Z}_{p^{\alpha}}\right)  \approx U^{k}\left(  \mathbb{Z}
_{p}\right)  \times U^{k-1}\left(  \mathbb{Z}_{p}\right)  \times\cdots\times
U^{k-\alpha+1}\left(  \mathbb{Z}_{p}\right)  \approx\left\{  0\right\}  .
\]
Therefore, $U^{k-\alpha+1}\left(  \mathbb{Z}_{p}\right)  \approx\left\{
0\right\}  .$

Conversely, let $\alpha<k$ and $U^{k-\alpha+1}\left(  \mathbb{Z}_{p}\right)
\approx\left\{  0\right\}  .$ Then $U^{s}(U^{k-\alpha+1}\left(  \mathbb{Z}
_{p}\right)  )$ $\approx\left\{  0\right\}  ,$ when $s=0,1,2\cdots,\alpha-1$.
That is,
\[
U^{k}\left(  \mathbb{Z}_{p}\right)  \approx U^{k-1}\left(  \mathbb{Z}
_{p}\right)  \approx\cdots\approx U^{k-\alpha+1}\left(  \mathbb{Z}
_{p}\right)  \approx\left\{  0\right\}  .
\]
Therefore, $U^{k}\left(  \mathbb{Z}_{p^{\alpha}}\right)  \approx\left\{
0\right\}  .$
\end{proof}

Next, we take an example to check if a given ring $U^{k}\left(  \mathbb{Z}%
_{p^{\alpha}}\right)  $ is a boolean ring or trivial. Let us suppose
$U^{k}\left(  \mathbb{Z}_{p^{\alpha}}\right)  =U^{9}(\mathbb{Z}_{23^{6}})$.
Starting with the necessary condition that $\alpha<k$ else it is not a boolean
neither a trivial group. Along we check the nature of $U^{k-\alpha+1}\left(
\mathbb{Z}_{p}\right)  =U^{9-6+1}(\mathbb{Z}_{23})=U^{4}(\mathbb{Z}_{23})$. We have

\begin{enumerate}
\item $U(\mathbb{Z}_{23})\approx\mathbb{Z}_{22}\approx\mathbb{Z}_{2}%
\times\mathbb{Z}_{11}$,

\item $U^{2}(\mathbb{Z}_{23})\approx U\left(  \mathbb{Z}_{11}\right)
\approx\mathbb{Z}_{10}\approx\mathbb{Z}_{2}\times\mathbb{Z}_{5}$,

\item $U^{3}(\mathbb{Z}_{23})\approx U\left(  \mathbb{Z}_{5}\right)
\approx\mathbb{Z}_{4}$ and

\item $U^{4}(\mathbb{Z}_{23})\approx U\left(  \mathbb{Z}_{4}\right)
\approx\mathbb{Z}_{2}$.
\end{enumerate}

Therefore, $U^{9}(\mathbb{Z}_{23^{6}})$ is a boolean ring with $U^{9}%
(\mathbb{Z}_{23^{6}})\approx U^{4}(\mathbb{Z}_{23})\approx\mathbb{Z}_{2}$.

\begin{theorem}
\label{d4}Let $p$ be an odd prime and let $k>0,\alpha>0$ and $t>0$. Then
$U^{k}\left(  \mathbb{Z}_{p^{\alpha}}\right)  $ is a boolean ring if and only
if $U^{k+t}\left(  \mathbb{Z}_{p^{\alpha+t}}\right)  $ is boolean ring.
Moreover, $U^{k}\left(  \mathbb{Z}_{p^{\alpha}}\right)  \approx U^{k+t}\left(
\mathbb{Z}_{p^{\alpha+t}}\right)  $.
\end{theorem}

\begin{proof}
Let $p$ be an odd prime and $U^{k}\left(  \mathbb{Z}_{p^{\alpha}}\right)  $ is
a boolean ring. Suppose $p\neq3$, from Lemma \ref{d1}, $\alpha<k,$ then
$\alpha+1<k+1.$ However, Theorem \ref{th1}, gives
\begin{align*}
U^{k+1}\left(  \mathbb{Z}_{p^{\alpha+1}}\right)   &  \approx U^{k+1}\left(
\mathbb{Z}_{p}\right)  \times U^{k}\left(  \mathbb{Z}_{p}\right)
\times\cdots\times U^{k+1-\left(  \alpha+1\right)  +1}\left(  \mathbb{Z}%
_{p}\right) \\
&  \approx U^{k+1}\left(  \mathbb{Z}_{p}\right)  \times U^{k}\left(
\mathbb{Z}_{p}\right)  \times\cdots\times U^{k-\alpha+1}\left(  \mathbb{Z}%
_{p}\right) \\
&  \approx U^{k+1}\left(  \mathbb{Z}_{p}\right)  \times U^{k}\left(
\mathbb{Z}_{p^{\alpha}}\right)  \text{.}%
\end{align*}
Also, from Corollary \ref{th2}, we have $U^{k}\left(  \mathbb{Z}_{p^{\alpha}%
}\right)  \approx U^{k-\alpha+1}\left(  \mathbb{Z}_{p}\right)  $ and are
boolean rings. Then $U^{\alpha}(U^{k-\alpha+1}\left(  \mathbb{Z}_{p}\right)
)=U^{k+1}\left(  \mathbb{Z}_{p}\right)  \approx\left\{  0\right\}  $.
Therefore, $U^{k+1}\left(  \mathbb{Z}_{p^{\alpha+1}}\right)  \approx
U^{k}\left(  \mathbb{Z}_{p^{\alpha}}\right)  $. Applying this isomorphism $t$
times, we obtain that $U^{k+t}\left(  \mathbb{Z}_{p^{\alpha+t}}\right)
\approx U^{k}\left(  \mathbb{Z}_{p^{\alpha}}\right)  .$

Conversely, let $U^{k+t}\left(  \mathbb{Z}_{p^{\alpha+t}}\right)  $ be a
boolean ring. Then by Theorem \ref{th22},%

\begin{align*}
U^{k+t}\left(  \mathbb{Z} _{p^{\alpha+t}}\right)   &  \approx U^{(k+t)-(\alpha
+t)+1}\left(  \mathbb{Z} _{p}\right) \\
&  =U^{k-\alpha+1}\left(  \mathbb{Z}_{p}\right)
\end{align*}

and are boolean rings. We may conclude that $U^{r}\left(  U^{k-\alpha
+1}\left(  \mathbb{Z} _{p}\right)  \right)  \approx\left\{  0\right\}  $ for
$r=1,2,\ldots,\alpha$. Thus we have
\begin{align*}
U^{k}\left(  \mathbb{Z}_{p^{\alpha}}\right)   &  \approx U^{k}\left(
\mathbb{Z} _{p}\right)  \times\cdots\times U^{k-\alpha+1}\left(  \mathbb{Z}%
_{p}\right) \\
&  \approx U^{k-\alpha+1}\left(  \mathbb{Z}_{p}\right)
\end{align*}

On the other hand, when $p=3,$ Corollary \ref{b7}, $\ $can be written
as\newline$U^{k}\left(  \mathbb{Z}_{3^{k}}\right)  \approx U^{k+t}\left(
\mathbb{Z} _{3^{k+t}}\right)  \approx\mathbb{Z}_{2}$.
\end{proof}

The following corollary is a direct consequence of Theorem \ref{d4}.

\begin{corollary}
Let $p$ be an odd prime and let $t>0,$ $k>0$ and $\alpha>0$. Then

\begin{enumerate}
\item $U^{k}\left(  \mathbb{Z}_{p^{\alpha}}\right)  $ is not a boolean ring if
and only if $U^{k+t}\left(  \mathbb{Z}_{p^{\alpha+t}}\right)  $ is not a
boolean ring.

\item $U^{k}\left(  \mathbb{Z}_{p^{\alpha}}\right)  $ is a trivial ring if and
only if $U^{k+t}\left(  \mathbb{Z}_{p^{\alpha+t}}\right)  $ is a trivial ring.

\item $U^{k}\left(  \mathbb{Z}_{p^{\alpha}}\right)  $ is nontrivial ring if
and only if $U^{k+t}\left(  \mathbb{Z}_{p^{\alpha+t}}\right)  $ is nontrivial ring.
\end{enumerate}
\end{corollary}

\begin{lemma}
\label{d9}Let $p$ be a prime and let $0\leq t\leq\alpha$. If $U^{k}\left(
\mathbb{Z}_{p^{\alpha}}\right)  \approx\left\{  0\right\}  $, then
$U^{k}\left(  \mathbb{Z}_{p^{t}}\right)  \approx\left\{  0\right\}  $.
\end{lemma}

\begin{proof}
Let $p$ be a prime and let $0\leq t\leq\alpha.$ Suppose $U^{k}\left(
\mathbb{Z}_{p^{\alpha}}\right)  \approx\left\{  0\right\}  $. By Theorem
\ref{d2}, we obtain that $\alpha<k.$ Now, since $0\leq t\leq\alpha,$ $0\leq
t<k$ and by Theorem \ref{th1}, we have
\[
U^{k}\left(  \mathbb{Z}_{p^{t}}\right)  \approx U^{k}\left(  \mathbb{Z}
_{p}\right)  \times U^{k-1}\left(  \mathbb{Z}_{p}\right)  \times\cdots\times
U^{k-t+1}\left(  \mathbb{Z}_{p}\right)  .
\]
Since $U^{k}\left(  \mathbb{Z}_{p^{\alpha}}\right)  \approx\left\{  0\right\}
$, Theorem \ref{d2} gives that $U^{k-\alpha+1}\left(  \mathbb{Z}_{p}\right)
\approx\left\{  0\right\}  $. However,\newline$U^{s}(U^{k-\alpha+1}\left(
\mathbb{Z} _{p}\right)  )$ $\approx\left\{  0\right\}  ,$ where
$s=0,1,2\cdots,t-1$. That is,
\[
U^{k}\left(  \mathbb{Z}_{p}\right)  \approx U^{k-1}\left(  \mathbb{Z}
_{p}\right)  \approx\cdots\approx U^{k-t+1}\left(  \mathbb{Z}_{p}\right)
\approx\left\{  0\right\}  .
\]
Therefore, $U^{k}\left(  \mathbb{Z}_{p^{t}}\right)  \approx\left\{  0\right\}
. $

Now, let $p=2.$ From Theorem \ref{b4}, $U^{k}\left(  \mathbb{Z}_{2^{\alpha}%
}\right)  \approx\left\{  0\right\}  $ if and only if $\alpha<2k.$ But $0\leq
t\leq\alpha<2k$. Therefore, $U^{k}\left(  \mathbb{Z}_{2^{t}}\right)
\approx\left\{  0\right\}  $.
\end{proof}

\begin{theorem}
\label{thdivisortrivial}Let $U^{k}\left(  \mathbb{Z}_{n}\right)
\approx\left\{  0\right\}  $. Then for all divisors $m$ of $n$, $U^{k}\left(
\mathbb{Z}_{m}\right)  \approx\left\{  0\right\}  $.
\end{theorem}

\begin{proof}
Let $n=$ $2^{\alpha}p_{1}^{\alpha_{1}}p_{2}^{\alpha_{2}}\cdot\cdot\cdot
p_{i}^{\alpha_{i}}$ be the decomposition of $n$ into product of distinct prime
powers and let $U^{k}\left(  \mathbb{Z}_{n}\right)  \approx\left\{  0\right\}
$. Suppose that $m$ be a divisor of $n.$ Then $m=2^{\alpha^{\prime}}%
p_{1}^{\alpha_{1}^{\prime}}p_{2}^{\alpha_{2}^{\prime}}\cdot\cdot\cdot
p_{i}^{\alpha_{i}^{\prime}}$, where$\,$ $0\leq\alpha_{i}^{\prime}\leq
\alpha_{i}$ and $0\leq\alpha^{\prime}\leq\alpha$. We have
\[
U^{k}\left(  \mathbb{Z}_{n}\right)  \approx U^{k}\left(  \mathbb{Z}%
_{2^{\alpha}}\right)  \times U^{k}\left(  \mathbb{Z}_{p_{1}^{\alpha_{1}}%
}\right)  \times\cdots\times U^{k}\left(  \mathbb{Z}_{p_{i}^{\alpha_{i}}}\right)
.
\]
Hence, $U^{k}\left(  \mathbb{Z}_{2^{\alpha}}\right)  ,U\left(  \mathbb{Z}
_{p_{1}^{\alpha_{1}}}\right)  ,\cdots$ and $U^{k}\left(  \mathbb{Z}
_{p_{i}^{\alpha_{i}}}\right)  $ are trivial. By the Lemma \ref{d9}, we obtain
that $U^{k}\left(  \mathbb{Z}_{2^{\alpha^{\prime}}}\right)  $, $U\left(
\mathbb{Z}_{p_{1}^{\alpha_{1}^{\prime}}}\right)  $,$\cdots$ and $U^{k}\left(
\mathbb{Z}_{p_{i}^{\alpha_{i}^{\prime}}}\right)  $ are trivial. Therefore,
\[
U^{k}\left(  \mathbb{Z}_{2^{\alpha^{\prime}}}\right)  \times U^{k}\left(
\mathbb{Z}_{p_{1}^{\alpha_{1}^{\prime}}}\right)  \times\cdots\times U^{k}\left(
\mathbb{Z}_{p_{i}^{\alpha_{i}^{\prime}}}\right)  \approx\left\{  0\right\}  ,
\]
and $U^{k}\left(  \mathbb{Z}_{m}\right)  $ is trivial.
\end{proof}

\begin{theorem}
Let $U^{k}\left(  \mathbb{Z}_{n}\right)  $ be a boolean ring. Then for all
divisors $m$ of $n$, $U^{k}\left(  \mathbb{Z}_{m}\right)  \approx\left\{
0\right\}  $ or $U^{k}\left(  \mathbb{Z}_{m}\right)  $ is a boolean ring.

\begin{proof}
Let $U^{k}\left(  \mathbb{Z}_{n}\right)  $ be a boolean ring. then
$U^{k+1}\left(  \mathbb{Z}_{n}\right)  \approx\left\{  0\right\}  $. From
Theorem \ref{thdivisortrivial} we have for all divisors $m$ of $n$,
$U^{k+1}\left(  \mathbb{Z}_{m}\right)  \approx\left\{  0\right\}  $. The later
leads to two possibilities either $U^{k}\left(  \mathbb{Z}_{m}\right)
\approx\left\{  0\right\}  $ or $U^{k}\left(  \mathbb{Z}_{m}\right)  $ is a
boolean ring.
\end{proof}
\end{theorem}

\section{Pratt's Tree and Decomposition of $U^{k}\left(
\mathbb{Z}
_{n}\right)  $}

V. Pratt in \cite{9} showed that short proofs of primality do exist, that is,
PRIMES is in NP where he introduced what so called Pratt certificate and Pratt
tree. The authors in \cite{10} discussed in details the dimensions of Pratt's
tree introduced in his paper. In this section, we relate Pratt's tree to the
complete structure of $U^{k}(%
\mathbb{Z}
_{n})$. The steps to find the decomposition $U^{k}(%
\mathbb{Z}
_{n})$ are similar to the step in determining the Pratt tree with same
structure and dimension, see Fig. 

of prime divisors of $n$, where all primes given in Pratt tree are used to
reach the decomposition of $U^{k}(%
\mathbb{Z}
_{n})$.

\noindent The primes in Pratt tree are the Prime chains $p_{1}\prec p_{2}%
\prec\cdots\prec p_{k}$ such that for which $p_{j+1}\equiv1\mod p_{j}$ in
other words, $p_{j}|p_{j+1}-1$ for each $j,$ see \cite{10}. By charting this
process, we find what is called a Lucas-Pratt tree \cite{11}. On the other
hand, we have%

\[
U^{k}\left(  \mathbb{Z}_{n}\right)  \approx U^{k}\left(  \mathbb{Z}%
_{p_{1}^{\alpha_{1}}}\right)  \times U^{k}\left(  \mathbb{Z}_{p_{2}%
^{\alpha_{2}}}\right)  \times\cdots\times U^{k}\left(  \mathbb{Z}_{p_{i}%
^{\alpha_{i}}}\right)  .
\]
If $p_{j}=2$ the decomposition is solved in Theorem \ref{b4}. While when
$p_{j}$ is an odd prime, it is clear from Corollary \ref{th2}, that to find
the decomposition of $U^{k}(\mathbb{Z}_{p_{j}^{\alpha_{j}}})$, we need to find
decomposition of $U^{i}\left(  \mathbb{Z}_{p_{j}-1}\right)  $, $i=1,2,..,k-1$.
which is also determined from the prime factors of $p_{j}-1$. which shall be
the Primes in the Pratt tree.

To illustrate the relation, we set $p_{j}=269$ and our aim is to determine the
decomposition of $U^{k}(%
\mathbb{Z}
_{269^{\alpha}})$,\ and relate it to Pratt Tree.

The Pratt Tree is obtained by the following steps:

\begin{description}
\item[Step 1] $p_{j}-1=268=2^{2}\times67$. First level in Pratt Tree is
$\left(  2,67\right)  $,

\item[Step 2] $67-1=66=2\times3\times11$ and Second Level in Pratt Tree is
$\left(  2,3,11\right)  ,$

\item[Step 3] $3-1=2$
also  $11-1=2\times5$ and thus Third Level in Pratt Tree is $\left(
2;2,5\right)  ,$

\item[Step 4] $5-1=2^{2}$ which give the Last level in the the Pratt Tree
$(2)$.
\end{description}

Next, we show the steps in determining the decomposition of $U^{k}(%
\mathbb{Z}
_{269^{\alpha}})$. Starting from Corollary \ref{th2}, which shows that the
decomposition of $U^{k}(%
\mathbb{Z}
_{269^{\alpha}})$ is determined from $U^{i}\left(  \mathbb{Z}_{p_{j}%
-1}\right)  $, $i=1,2,..,k-1$. Thus
\[
U^{i}\left(  \mathbb{Z}_{p_{i}-1}\right)  =U^{i}(%
\mathbb{Z}
_{268})\approx U^{i}(%
\mathbb{Z}
_{2^{2}})\times U^{i}(%
\mathbb{Z}
_{67}),
\]
$i=1,2,..,k-1$. Thus we get in this decomposition the First Level of Pratt
Tree $(2,67)$. Now, the decomposition of $U^{i}(%
\mathbb{Z}
_{2^{2}})$ can be determined from Theorem \ref{b4}, thus next we need to
determine the decomposition of $U^{i}(%
\mathbb{Z}
_{67})$.

Having $67-1=66=2\times3\times11$ thus we get
\[
U^{i}(%
\mathbb{Z}
_{67})\approx U^{i-1}(%
\mathbb{Z}
_{66})\approx U^{i-1}(%
\mathbb{Z}
_{2})\times U^{i-1}(%
\mathbb{Z}
_{3})\times U^{i-1}(%
\mathbb{Z}
_{11}),
\]
we get in this decomposition the Second Level of Pratt Tree $(2,3,11)$. Next,
we need the decomposition of $U^{i-1}(%
\mathbb{Z}
_{3})\ $\ and $U^{i-1}(%
\mathbb{Z}
_{11})$. In the same manner, we get
\[
U^{i-1}(%
\mathbb{Z}
_{3})\approx U^{i-2}(%
\mathbb{Z}
_{2})\text{ and }U^{i-1}(%
\mathbb{Z}
_{11})\approx U^{i-2}(%
\mathbb{Z}
_{2})\times U^{i-2}(%
\mathbb{Z}
_{5}).
\]
This is Third Level in Pratt Tree $(2)$ and $(2,5)$.

Finally with the last decomposition of
\[
U^{i-2}(%
\mathbb{Z}
_{5})\approx U^{i-3}(%
\mathbb{Z}
_{2^{2}})
\]
giving Last Level of Pratt Tree that is $(2)$.

As a consequence that all the decompositions will reach eventually to $U^{k}(%
\mathbb{Z}
_{2^{\alpha}})$ which is solved in Theorem \ref{b4}. Thus the decomposition of
$U^{k}(%
\mathbb{Z}
_{n})$ can be obtained.

\begin{figure}[h!]
	\centering
	\begin{subfigure}{.5\textwidth}
		\centering
		\includegraphics[width=.7\linewidth]{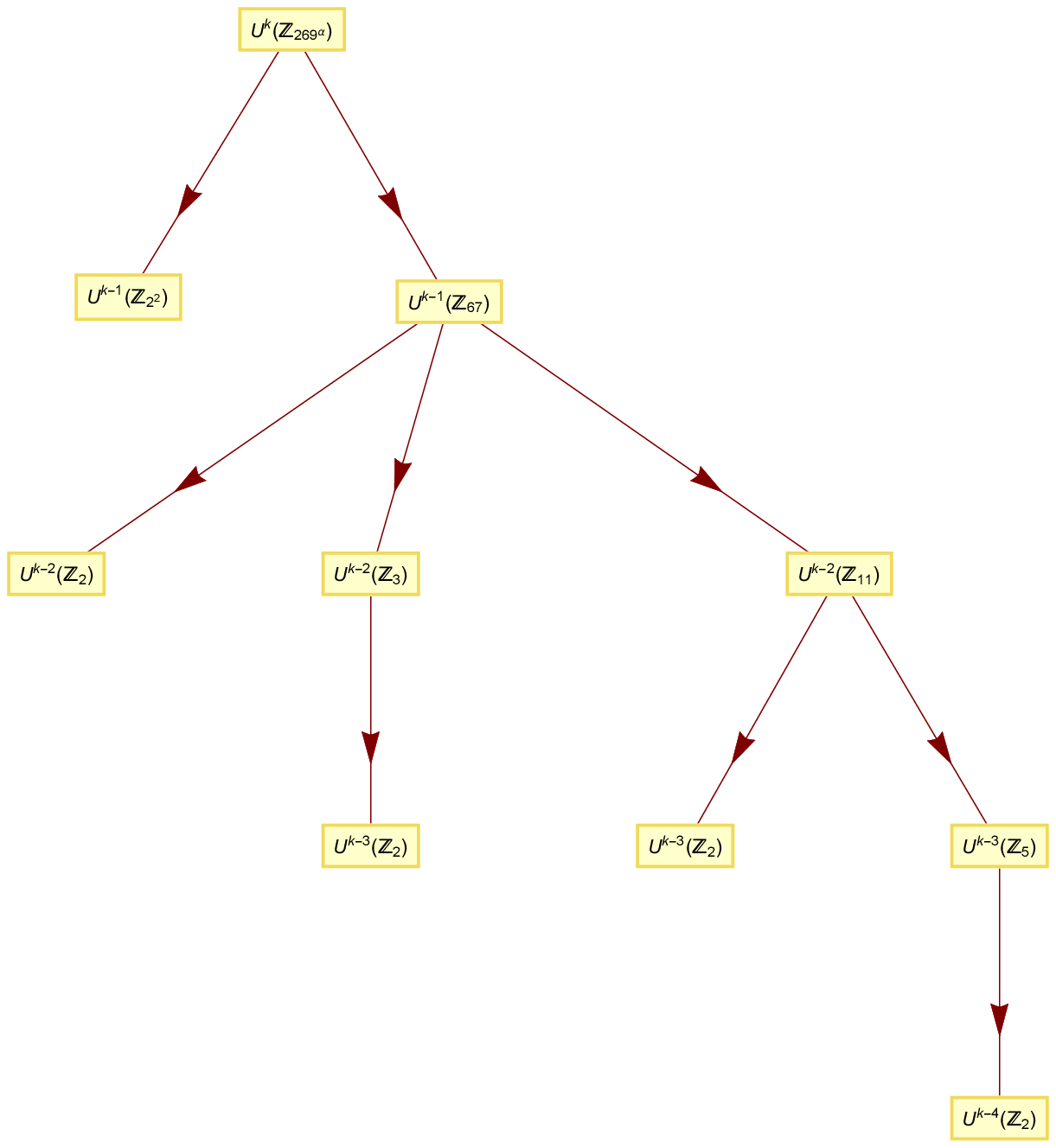}
		\caption{ Decomposition of $U^k(\mathbb{Z}_{269^\alpha})$}
	\label{fig:sub1}
	\end{subfigure}%
	\begin{subfigure}{.5\textwidth}
		\centering
		\includegraphics[width=.7\linewidth]{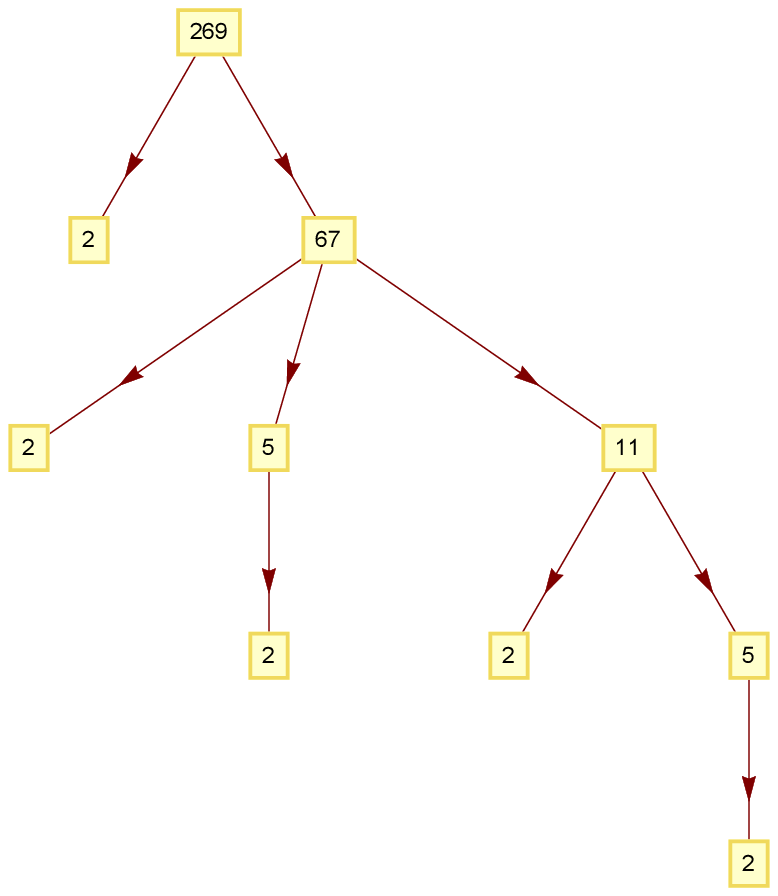}
		\caption{Pratt Tree of $269$}
		\label{fig:sub2}
	\end{subfigure}
	\caption{Equivalence between the decomposition $U^k(\mathbb{Z}_{269^\alpha})$ and Pratt Tree.}
	\label{fig:test}
\end{figure}

%


\end{document}